\documentclass[11pt,reqno]{amsart}
\usepackage[a4paper,bindingoffset=0.1in,left=1in,right=1in,top=1in,bottom=1in,footskip=.25in]{geometry}
\usepackage{indentfirst,amssymb,amsmath,amsthm}    
\usepackage{newtxtext,newtxmath}
\usepackage{setspace}
\usepackage{times}
\usepackage[utf8]{inputenc}
\usepackage[T1]{fontenc}
\usepackage{verbatim}
\usepackage{hyperref}
\usepackage{csquotes}
\hypersetup{colorlinks=true, linkcolor=blue,citecolor=blue, urlcolor=blue}
\newcommand{\divides}{\mid}

\DeclareMathOperator{\ord}{ord}
\DeclareMathOperator{\dime}{dim}
\DeclareMathOperator{\pideg}{PI-deg}

\DeclareMathOperator{\diagonal}{diag}

\DeclareMathOperator{\Id}{Id}
\usepackage{mathtools}

\numberwithin{equation}{section}
 \newtheorem*{Theorem D}{Theorem D}
\newtheorem*{Theorem C}{Theorem C}
\newtheorem*{Theorem A}{Theorem A}
\newtheorem*{Theorem B}{Theorem B}
\newtheorem{theo}{Theorem}[section]
\newtheorem{defi}[theo]{Definition}
\newtheorem{lemm}[theo]{Lemma}
\newtheorem{rema}[theo]{Remark}
\newtheorem{coro}[theo]{Corollary}
\newtheorem{prop}[theo]{Proposition}

\begin{document}

\setcounter{page}{1} 
\pagenumbering{arabic}

\title[Representations of $U_q^+(B_2)$]{$U_q^+(B_2)$ and its representations}
\author [S Bera \ S Mukherjee]{Sanu Bera \ \ Snehashis Mukherjee}
\address {\newline  Sanu Bera$^1$, \newline Gandhi Institute of Technology and Management (GITAM), Hyderabad, \newline Rudraram, Patancheru mandal. Hyderabad-502329. Telangana, India.\newline  Snehashis Mukherjee$^2$, 
\newline Indian Institute of Technology (IIT), Kanpur, \newline Kalyanpur. Kanpur -208 016, Uttar Pradesh, India. 
}
\email{\href{mailto:sanubera6575@gmail.com}{sanubera6575@gmail.com$^1$}; \href{mailto:tutunsnehashis@gmail.com}{tutunsnehashis@gmail.com$^2$}}

\subjclass[2020]{16D60, 16D70, 16S85}
\keywords{$U_q^+(B_2)$, Polynomial Identity algebra, Simple modules}
\begin{abstract}  
In this article we investigate the algebra $U_q^+(B_2)$. Assume that $q$ is a primitive $m$-th root of unity with $m \geq 5$. We prove that $U_q^+(B_2)$ becomes a Polynomial Identity (PI) algebra. It was previously known that for such algebras the simple modules are finite-dimensional with dimension at most the PI degree. We determine the PI degree of $U_q^+(B_2)$ and we classify up to isomorphism the simple $U_q^+(B_2)$-modules. We also find the center of $U_q^+(B_2)$.
 \end{abstract}
\maketitle
\section{Introduction} 
Unless otherwise stated throughout the paper a module means a right module,  $\mathbb{K}$ is an algebraically closed field, $q$ is a primitive $m(\geq 5)$-th root of unity and  \begin{align*}
 l:=\begin{cases}
     m,& \text{if $m$ is odd},\\
     \displaystyle\frac{m}{2},& \text{if $m$ is even}.
 \end{cases}
\end{align*}\\
Let $\mathfrak{g}$ be a complex simple algebra of type $B_2$. This article aims to describe all possible simple modules over the algebra $U_q^+(B_2)$ and also compute its center. Recall that  $U_q^+(B_2)$ is the $\mathbb{K}$-algebra generated by two indeterminates $e_1$ and $e_2$ subject to the quantum Serre relations:
\begin{align*}
    &e_1^2e_2-(q^2+q^{-2})e_1e_2e_1+e_2e_1^2=0\\
    &e_2^3e_1-(q^2+1+q^{-2})e_2^2e_1e_2+(q^2+1+q^{-2})e_2e_1e_2^2-e_1e_2^3=0.
\end{align*}
The universal enveloping algebras of semisimple and solvable Lie algebras are relatively well studied compared to their nilpotent counterparts. However, techniques for analyzing nilpotent cases remain underdeveloped. The same holds true for the quantizations of universal enveloping algebras and quantum groups. The algebra $U_q^+(B_2)$ serves as a fundamental yet highly nontrivial example of a quantization in the nilpotent setting.
\par For most noncommutative algebras, explicitly describing their simple modules and classifying their isomorphism classes is a challenging problem. In this paper, we provide a solution to this problem for $U_q^+(B_2)$, showing that its simple modules have a complex structure. Additionally, we develop techniques that may prove useful in studying similar algebras of ‘small’ rank.
\par So far, whatever studies have been done for $U_q^+(B_2)$ are entirely for the generic case, i.e., when $q$ is not a root of unity. In the classical case, a theorem of Dixmier asserts that the simple factor algebras of Gelfand-Kirillov dimension 2 of the positive part $U^+(B_2)$ of the enveloping algebra of $B_2$ are isomorphic to the first Weyl algebra. In order to obtain some new quantized analogues of the first Weyl algebra, Launois \cite{lau} explicitly described the prime and primitive spectra of the positive part $U_q^+(B_2)$ of the quantized enveloping algebra of $B_2$ and then studied the simple factor algebras of Gelfand-Kirillov dimension $2$ of $U_q^+(B_2)$. In particular, they showed that the centers of such simple factor algebras are reduced to the ground field $\mathbb{C}$ and computed their group of invertible elements. These computations allowed them to prove that the automorphism group of $U_q^+(B_2)$ is isomorphic to the torus $(\mathbb{C}^*)^2$, as conjectured by Andruskiewitsch and Dumas.
\par Another significant contribution to this area was made by A. Mériaux \cite{meri}. For each reduced decomposition of the longest element 
$w_0$ in the Weyl group $\mathcal{W}$, there exists a corresponding Poincaré–Birkhoff–Witt (PBW) basis of the quantized enveloping algebra $U_q^+(\mathfrak{g})$. The theory of deleting-derivation can then be applied to this iterated Ore extension. Specifically, for every decomposition of $w_0$, this theory establishes a bijection between the set of prime ideals in $U_q^+(\mathfrak{g})$ that are invariant under a natural torus action and certain combinatorial structures known as Cauchon diagrams.
\par In his paper, Mériaux provided an algorithmic description of these Cauchon diagrams when the chosen reduced decomposition of $w_0$ aligns with a good ordering (as defined by Lusztig (1990) \cite{lus}) of the set of positive roots. This description is based on constraints derived from Lusztig’s admissible planes [Lus90], where each admissible plane imposes a set of conditions that a diagram must satisfy to be classified as a Cauchon diagram. Additionally, he explicitly described the set of Cauchon diagrams for particular reduced decompositions of $w_0$ across all possible Lie types. In each case, he verified that the number of Cauchon diagrams always matches the cardinality of $\mathcal{W}$. In a subsequent work by Cauchon and Mériaux (2008) \cite{cm}, these results were leveraged to establish that Cauchon diagrams correspond canonically to the positive subexpressions of $w_0$. Consequently, the findings of this paper also provided an algorithmic method for describing the positive subexpressions of any reduced decomposition of $w_0$ associated with a good ordering.
\par Now, let us shift our focus to the non-generic case, where $q$ is a root of unity. It is important to note that the Derivation Deletion algorithm is only applicable when $q$ is not a root of unity. Launois et al. \cite{lau2} extended both this algorithm and Cauchon's canonical embedding to a broad class of quantum algebras. Their extension applies to iterated Ore extensions over a field, provided certain suitable conditions are met. These conditions encompass Cauchon's original framework while also accommodating cases where $q$ is a root of unity.
\par The extended algorithm constructs a quantum affine space $A'$ from the original quantum algebra $A$ through a sequence of variable changes within the division ring of fractions Frac$(A$). The canonical embedding maps a completely prime ideal $P$ of $A$ to a completely prime ideal $Q$ of $A'$, ensuring that when $A$ is a PI algebra, the PI-degree satisfies the relation
$\pideg{A/P}=\pideg{A'/Q}$. In the case where the quantum parameter is a root of unity, an explicit formula can be stated for the PI-degree of completely prime quotient algebras. The paper concludes by presenting a method for constructing a maximum-dimensional irreducible representation of $A/P$, given a suitable irreducible representation of $A'/Q$ when $A$ is a PI algebra. 
\par To illustrate their algorithm, they constructed an irreducible representation of a quotient of $U_q^+(B_2)$ by the ideal $\langle z' \rangle$, where $z'$ is a central element of $U_q^+(B_2)$ \cite[Example 5.1]{lau2}. This construction involved applying the deleting derivations algorithm to both $U_q^+(B_2)$ and $U_q^+(B_2)/\langle z' \rangle$, computing the image of $\langle z' \rangle$ under the canonical embedding, and determining the PI degree of 
$U_q^+(B_2)/\langle z' \rangle$.
\par As mentioned earlier classifying simple modules for infinite dimensional algebras is a difficult problem. Much of the recent progress has focused on the classification of simple modules over generalized Weyl algebras (see \cite{ba1, ba2, ba3, ba5}) and Ore extensions with Dedekind rings as coefficient rings (see \cite{ba4}).
\par When dealing with a complex algebra like $U_q^+(B_2)$ with numerous defining relations, determining its all simple modules is not straightforward—it is unclear where to begin or which direction to take. The approach we adopt in this paper is to identify certain subalgebras of $U_q^+(B_2)$. Specifically, we cover $U_q^+(B_2)$ with a chain of relatively large subalgebras, which are generalized Weyl algebras. These subalgebras possess nontrivial normal elements that we explicitly identify using Proposition \ref{GWA}. At each step, while the individual elements become more intricate, the defining relations simplify, tending toward a more commutative (i.e., 
$q$-commutative or fully commutative) structure. Classifying simple modules for these subalgebras is more manageable than for $U_q^+(B_2)$ itself. Once we construct and classify all simple modules of these subalgebras, we extend them to simple modules of $U_q^+(B_2)$ under certain conditions. This significantly simplifies the task of identifying the remaining simple modules of $U_q^+(B_2)$. We believe that this method can be applied to classify all simple modules of various complex algebras with multiple defining relations.\\
\textbf{Arrangement:} The article is structured as follows. In Section 2, we build on the results of Andruskiewitsch and Dumas \cite[Section 3.1.2]{dum} to introduce a new set of generators for $U_q^+(B_2)$. We set \[e_3=e_1e_2-q^2e_2e_1, \ z=e_2e_3-q^2e_3e_2\]This, in turn, allows us to express $U_q^+(B_2)$ as an iterated skew polynomial ring. We then discuss some preliminary properties of $U_q^+(B_2)$ and establish that it is a polynomial identity (PI) algebra with $\pideg{U_q^+(B_2)}=l$ (see Theorem \ref{PI}). This result guarantees that all simple modules of $U_q^+(B_2)$ are finite-dimensional, with their dimensions bounded above by $l$, the PI degree. Note that in this article, our primary focus is on constructing $z$-torsion-free simple  $U_q^+(B_2)$-modules. Since $z$ is a central element, it either acts invertibly or as zero on any simple module $N$ of 
$U_q^+(B_2)$. If $Nz=0$, then $N$ becomes a simple module over $U_q^+(B_2)/\langle z \rangle \cong U_q^+(sl_3)$, which is known as the quantum Heisenberg algebra. The classification of simple modules over $U_q^+(sl_3)$ at roots of unity has already been established in \cite{sanu}.
\par In Section 3, we construct a "good" subalgebra $\textbf{B}$ using the concept of generalized Weyl algebras introduced by V. Bavula. We observe that 
$\textbf{B}$ is also a PI algebra with the same PI degree as $U_q^+(B_2)$. As mentioned earlier, $\textbf{B}$ possesses additional properties that significantly simplify the classification of its simple modules. Notably, we find that the element $e_3$ in $\textbf{B}$ is a normal element. Sections 4 and 5 focus on the construction and classification of all $e_3$-torsion-free simple modules over $\textbf{B}$ (see Theorem \ref{main1}). In Section 6, we establish a one-to-one correspondence between $e_3$ -torsion-free simple $\textbf{B}$-modules and simple modules of $U_q^+(B_2)$ in which $e_3$  acts invertibly (see Theorem \ref{itd}). This result enables us to construct all simple modules of $U_q^+(B_2)$ where $e_3$ acts invertibly.
\par In Section 7, we shift our focus to the construction and classification of simple $U_q^+(B_2)$ -modules in which $e_3$ acts nilpotently, thereby completing the classification problem. Finally, in Section 8, we further apply the concept of generalized Weyl algebras to identify suitable central elements of $U_q^+(B_2)$ and compute its center.
\section{$U_q^+(B_2)$ and some related results} Recall that the monomials $\{z^ie_3^je_1^ke_2^l:i,j,k,l \in \mathbb{Z}^{\geq 0}\}$ form a PBW-basis of $U_q^+(B_2)$ \cite[Section 3.1.2]{dum}. Hence $U_q^+(B_2)$ is the $\mathbb{K}$-algebra generated by $e_1, e_2, e_3$ and $z$ with the relations :
\[e_iz=ze_i \ \forall i=1,2,3, \ e_1e_3=q^{-2}e_3e_1, e_2e_3=q^2e_3e_2+z, \ e_2e_1=q^{-2}e_1e_2-q^{-2}e_3.\] In other words, $U_q^+(B_2)$ is an iterated Ore extension that we can write as follows :
\[U_q^+(B_2)=\mathbb{K}[z,e_3][e_1; \sigma_1][e_2;\sigma_2,\delta_2]\] where $\sigma_1$ denotes the automorphism of $\mathbb{K}[z,e_3]$ defined by $\sigma_1(z)=z$ and $\sigma_1(e_3)=q^{-2}e_3$, $\sigma_2$ denotes the automorphism of $\mathbb{K}[z,e_3][e_1; \sigma_1]$ defined by $\sigma_2(z)=z$, $\sigma_2(e_3)=q^{2}e_3$ and $\sigma_2(e_1)=q^{-2}e_1$, and where $\delta_2$ denotes the $\sigma_2$-derivation of $\mathbb{K}[z,e_3][e_1; \sigma_1]$ defined by $\delta_2(z)=0$, $\delta_2(e_3)=z$ and $\delta_1(e_1)=-q^{-2}e_3$.
\begin{lemm}
    The following identities hold in $U_q^+(B_2)$.
    \begin{itemize}
        \item[1.] $e_2e_3^k=q^{2k}e_3^ke_2+\displaystyle\frac{q^{2k}-1}{q^2-1}ze_3^{k-1}$.
        \item[2.] $e_2^ke_3=q^{2k}e_3e_2^k+\displaystyle\frac{q^{2k}-1}{q^2-1}ze_2^{k-1}$.
        \item[3.] $e_2e_1^{k}=q^{-2k}e_1^ke_2-q^{-2}\displaystyle\frac{q^{-4k}-1}{q^{-4}-1}e_3e_1^{k-1}$.
        \item[4.] $e_2^ke_1=q^{-2k}e_1e_2^k-q^{-2}\displaystyle\frac{q^{2k}-q^{-2k}}{q^2-q^{-2}}e_3e_2^{k-1}-q^{-2(k-1)}\displaystyle\frac{(q^{2k}-1)(q^{2(k-1)}-1)}{(q^4-1)(q^2-1)}ze_2^{k-2}, \ \ k\geq 2$. 
    \end{itemize}
\end{lemm}
\begin{proof}
    These commutation relations can be easily obtained by induction.
\end{proof}
\begin{coro}\label{cen4}
    The elements $e_1^l,e_2^l,e_3^l$ and $z$ are central in $U_q^+(B_2)$.
\end{coro}
\subsection{Polynomial Identity Algebras} In the roots of unity setting, we will show that $U_q^+(B_2)$ becomes a polynomial identity algebra. This sufficient condition on the parameters to be PI algebra is also necessary.
\begin{prop}
    The algebra $U_q^+(B_2)$ is a PI algebra if and only if $q$ is a root of unity.
\end{prop}
\begin{proof}
   Consider the polynomial algebra $P= \mathbb{K}[e_1^{l},e_2^{l},e_3^{l},z]$. It is clear from Corollary \ref{cen4} that $P$ is a central subalgebra of $U_q^+(B_2)$. Note $U_q^+(B_2)$ becomes a finitely generated module over $P$. Hence the assertion follows from proposition \cite[Corollary 13.1.13]{mcr}.
\end{proof}
Kaplansky's Theorem has a striking consequence in the case of a prime affine PI algebra over an algebraically closed field.
\begin{prop}\label{brg}\emph{(\cite[Theorem I.13.5]{brg})}
Let $A$ be a prime affine PI algebra over an algebraically closed field $\mathbb{K}$ and $V$ be a simple $A$-module. Then $V$ is a finite-dimensional vector space over $\mathbb{K}$ with $\dime_{\mathbb{K}}(V)\leq \pideg (A)$.
\end{prop} 
This result provides the important link between the PI degree of a prime affine PI algebra over an algebraically closed field and its irreducible representations. Moreover, the upper bound PI-deg($A$) is attained for such an algebra $A$ (cf. \cite[Lemma III.1.2]{brg}). 
\begin{rema}\label{fdb}
From the above discussion, it is quite clear that each simple $U_q^+(B_2)$-module is finite-dimensional and can have dimension at most $\pideg {U_q^+(B_2)}$. Therefore the calculation of PI degree for $U_q^+(B_2)$ is of substantial importance.
\end{rema}
\subsection{PI degree for $U_q^+(B_2)$} In this subsection our main focus will be to compute the PI degree of $U_q^+(B_2)$. For that we construct an embedding of the algebra $U_q^+(B_2)$ into a quantum torus. The embedding shall enable us to compute the PI degree of $U_q^+(B_2)$ by computing the PI degree of that quantum torus (See \cite[Corollary I.13.3]{brg}). The point is that the algebra $U_q^+(B_2)$ has a Goldie quotient ring, which we shall denote by $Q\left(U_q^+(B_2)\right)$. Within the Goldie quotient ring of $U_q^+(B_2)$, let us define the following new variables
\[X_1=e_1, X_2=e_3, X_3=z, X_4=e_3^{-1}z'e_1^{-1}\] where \[z':=e_1\left(z+(q^2-1)e_3e_2\right)-q^4\left(z+(q^2-1)e_3e_2\right)e_1.\]
\begin{prop}\cite[Proposition 1.4]{xt}
    The algebra $Q:=\mathbb{K}_q[X_1^{\pm1}, X_2^{\pm1}, X_3^{\pm1}, X_4^{\pm1}]$ is a quantum torus. Moreover, the linear map $\mathcal{I}: U_q^+(B_2) \rightarrow Q $ defined by 
    \[\mathcal{I}(e_1)=X_1, \mathcal{I}(e_3)=X_2, \mathcal{I}(z)=X_3, \mathcal{I}(e_2)=\lambda\left(X_4+(q^4-1)X_2^{-1}X_3+(q^{-2}-1)X_2X_1^{-1}\right)\]  where \[\lambda:=\displaystyle\frac{q}{(q^2-q^{-2})(q-q^{-1})}\]is an algebra monomorphism.
\end{prop}
Note in $Q$, the following relations hold :
\[X_1X_2=q^{-2}X_2X_1, X_1X_3=X_3X_1, X_1X_4=q^2X_4X_1, X_2X_3=X_3X_2, X_2X_4=q^{-2}X_4X_2, X_3X_4=X_4X_3.\] Then from \cite[Corollary I.14.I]{brg} we have 
\[\pideg(U_q^+(B_2))=\pideg(Q)\] where the associated antisymmetric matrix $H$ with $Q$ is a $4 \times 4$ matrix \begin{center}
    $\begin{pmatrix}
0 & 2 & -2 & 0\\
-2 & 0 & 2 & 0\\
2 & -2 & 0 & 0\\
0 & 0 &0 & 0\\
\end{pmatrix}$
\end{center}
Now we can easily verify that $H$ is similar to the integral matrix 
\[H'=\diagonal\left(\begin{pmatrix}
0&2\\
-2&0
\end{pmatrix},0_2\right).\] Hence they share the same invariant factors. Hence by \cite[Lemma 5.7]{ar}, $\pideg\left(U_q^+(B_2)\right)=\ord(q^2)=l$.
\begin{theo}\label{PI}
    $U_q^+(B_2)$ is a PI algebra and $\pideg\left(U_q^+(B_2)\right)=l$.
\end{theo}
\section{A \enquote{Good} subalgebra $\textbf{B}$}
It is more straightforward to begin by constructing the representations of a subalgebra of $U_q^+(B_2)$. This subalgebra has the advantage that, while its commutation relations are considerably simpler than those of $U_q^+(B_2)$, its representations can be directly utilized to construct representations of $U_q^+(B_2)$. In this section, we introduce this subalgebra, which we will denote by $\textbf{B}$, and examine its structure in detail. We begin this section by defining the generalized Weyl algebra introduced by V. Bavula, a concept crucial in identifying this "good" subalgebra of $U_q^+(B_2)$.
\begin{defi}\cite{ba3}
    Let $D$ be a ring, $\sigma$ be an automorphism of $D$
and $a$ is an element of the center of $D$. The generalized Weyl algebra $A:=D[X,Y;\sigma,a]$ is a ring generated by $D, X$ and $Y$ subject to the defining relations:
\[X\alpha=\sigma(\alpha)X,~ Y\alpha=\sigma^{-1}(\alpha)Y~\text{for all $\alpha \in D$},~ YX=a,~XY=\sigma(a)\]
\end{defi}
In this article, we will show that many subalgebras of $U_q^+(B_2)$ are GWAs.
\begin{defi}\cite{ba3}
    Let $D$ be a ring and $\sigma$ be its automorphism. Suppose that elements $b$ and $\rho$ belong to the centre of the ring $D$, $\rho$ is invertible and $\sigma(\rho)=\rho$.
Then $E=D[X,Y;\sigma,\rho,b]$ is a ring generated by $D, X$ and $Y$ subject to the defining relations:
\[X\alpha=\sigma(\alpha)X,~ Y\alpha=\sigma^{-1}(\alpha)Y~\text{for all $\alpha \in D$},~XY-\rho YX=b.\]
\end{defi}
The next proposition shows that the rings $E$ are GWAs and under a (mild) condition they have a \enquote{canonical} normal element.
\begin{prop}\label{GWA}\cite[Lemma 1.3, Corollary 1.4, Corollary 1.6]{ba3}
    Let $E=D[X,Y;\sigma,\rho,b]$. Then 
    \begin{enumerate}
        \item The following statements are equivalent
        \begin{itemize}
            \item[(a)] $C=\rho(YX+\alpha)=XY+\sigma(\alpha)$  is a normal element in $E$ for some central element $\alpha \in D$.
            \item[(b)] $\rho\alpha-\sigma(\alpha)=b$ for some central element $\alpha \in D$.
        \end{itemize}
        \item  If one of the equivalent conditions of statement 1 holds then the ring $E=D[C][X, Y,\sigma,a=\rho^{-1}C-\alpha]$  is a GWA where $\sigma(C)=\rho C$. 
    \end{enumerate}
    Moreover if $\rho=1$ then the element $C$ is central in $E$.
\end{prop}
\textbf{The algebra $\textbf{A}$ is a GWA}. Let $\textbf{A}$ be the subalgebra of $U_q^+(B_2)$ which is generated by the elements $e_2$, $e_3$ and $z$. Then the elements satisfy the defining relations
\[e_2z=ze_2,~e_3z=ze_3,~e_2e_3-q^2e_3e_2=z.\] Next, by Proposition \ref{GWA} we conclude that $\textbf{A}$ is a GWA and $\textbf{A}=\mathbb{K}[z][e_2,e_3;\sigma,\rho=q^{2},b=z]$ where $\sigma=\Id_{\mathbb{K}[z]}$. The element $\displaystyle\frac{z}{q^2-1} \in \mathbb{K}$ is a solution of equation $q^{2}\alpha-\sigma(\alpha)=z$. Hence, by Proposition \ref{GWA}, the element \begin{equation}\label{cen1}\tilde{z}:=e_2e_3+\frac{z}{q^2-1}\end{equation}
is a normal element of $\textbf{A}$. Note that we have the following relations of $\tilde{z}$ with the generators of $U_q^+(B_2)$.
\[\tilde{z}e_2=q^{-2}e_2\tilde{z},~\tilde{z}e_3=q^2e_3\tilde{z},~\tilde{z}z=z\tilde{z},~\tilde{z}e_1=e_1\tilde{z}-e_3^2.\]
\textbf{A \enquote{good} subalgebra $B$:} Let us consider the subalgebra $\textbf{B}$ of $U_q^+(B_2)$ generated by elements $e_1,e_3, z$ and $\tilde{z}$. The relations between the generators are as follows \[e_1e_3=q^{-2}e_3e_1,e_1z=ze_1,~e_3z=ze_3,\tilde{z}z=z\tilde{z},~\tilde{z}e_3=q^2e_3\tilde{z},~\tilde{z}e_1=e_1\tilde{z}-e_3^2.\]
The following theorem will be useful.
\begin{theo}
    The following identities hold in $\textbf{B}$.
    \begin{enumerate}
        \item $e_1^a\tilde{z}=\tilde{z}e_1^a+\displaystyle\frac{1-q^{-4a}}{1-q^{-4}}e_3^2e_1^{a-1}$.
        \item $e_1\tilde{z}^a=\tilde{z}^ae_1+\displaystyle\frac{1-q^{4a}}{1-q^{4}}e_3^2\tilde{z}^{a-1}$.
    \end{enumerate}
\end{theo}
\begin{proof}
    The above identities can be obtained by induction of the defining relations of the algebra $\textbf{B}$. See the appendix for the calculations.
\end{proof}
\begin{coro}\label{cen2}
    The elements $z,e_1^l,e_3^l$ and $\tilde{z}^l$ are central in $\textbf{B}$.
\end{coro}
Then we have the following result.
\begin{theo}\label{finite2}
    $\textbf{B}$ is a prime polynomial identity algebra.
\end{theo}
\begin{proof}
    Let $C= \mathbb{K}[z,e_1^l, e_3^l, \tilde{z}^l]$ be the polynomial algebra. The Corollary \ref{cen2} clearly shows that $C$ is a central subalgebra of $\textbf{B}$. Observe that $\textbf{B}$ is a finitely generated module over $C$. Thus, the conclusion is derived from Proposition \cite[Corollary 13.1.13]{mcr}.
\end{proof}
The algebra $\textbf{B}$ has an iterated skew polynomial expression
\[\mathbb{K}[e_3,z][e_1, \tau_1][\tilde{z},\tau_2;\delta_2]\] where $\tau_1$ is an automorphism of $\mathbb{K}[e_3,z]$ defined by $\tau_1(e_3)=q^{-2}e_3$ and $\tau_1(z)=z$, $\tau_2$ is the automorphism of $\mathbb{K}[e_3,z][e_1, \tau_1]$ such that $\tau_2(e_3)=q^2e_3$, $\tau_2(e_1)=e_1$ and $\tau_2(z)=z$. Finally, $\delta_2$ is a $\tau_2$ derivative of $\mathbb{K}[e_3][e_1, \tau_1]$ defined by $\delta_2(e_1)=e_3^2$, $\delta_2(e_3)=0$ and $\delta_2(z)=0$. Note that
\[\delta_2\tau_2(e_1)=e_3^2=q^{-4}\left(\tau_2\delta_2(e_1)\right).\] This holds trivially if $e_1$ is replaced by $e_3$ and $z$. So, the pair $(\tau_2,\delta_2)$ is a $q^{-4}$-derivation on $\mathbb{K}[e_3,z][e_1, \tau_1]$. Moreover, we can check that all the hypotheses of the derivation
erasing process in \cite[Hypothesis 1.2, Section 2]{lau2} are satisfied by the skew polynomial presentation of the PI algebra $\textbf{B}$. Hence, it follows that
\[\pideg(\textbf{B})=\pideg \mathcal{O}_{\Lambda}(\mathbb{K}^4).\] where the $4\times 4$-matrix of relations of $\Lambda$ is 
 \begin{center}
    $\begin{pmatrix}
1 & 1 & 1 & 1\\
1 & 1 & q^2 & q^{-2}\\
1 & q^{-2} & 1 & 1 \\
1 & q^2 & 1 & 1 \\
\end{pmatrix}.$
\end{center} So the integral matrix $H$ associated with $\Lambda$ is 
\begin{center}
    $\begin{pmatrix}
        0 & 0& 0& 0\\
        0 & 0 & 2 & -2\\
        0 & -2 & 0 & 0\\
        0 & 2 & 0 & 0
    \end{pmatrix}$.
\end{center}
Now we can easily verify that $H$ is similar to the integral matrix 
\[H'=\diagonal\left(\begin{pmatrix}
0&2\\
-2&0
\end{pmatrix},\begin{pmatrix}
0&0\\
0&0
\end{pmatrix}\right).\] Hence they share the same invariant factors. So by \cite[Lemma 5.7]{ar}, $\pideg(\textbf{B})=\ord(q^2)=l$. Hence from Theorem \ref{finite2} and the above discussion, we have the following
\begin{theo}\label{pideg2}
   $\textbf{B}$ is a PI ring and $\pideg(\textbf{B})=l$.
\end{theo}
From Theorem $\ref{pideg2}$ and \cite[Theorem I.13.5]{brg}, we conclude if $M$ is a $\textbf{B}$-simple module, then $\dime_{\mathbb{K}}M \leq \pideg(\textbf{B})=\ord(q^2)$.
\section{Construction of $e_3$-torsionfree simple modules over $\textbf{B}$} We begin this section by defining what it means for an element of an algebra to be normal. A non-zero element $x$ of an algebra $A$ is termed a normal element if it satisfies the condition $xA = Ax$. It is evident that if $x$ is a normal element of $A$, then the set $\{x^i:~ i \geq 0\}$ forms an Ore set generated by $x$. The following lemma is straightforward.

\begin{lemm}\label{itn}
Let $A$ be an algebra, $x$ a normal element of $A$, and $M$ a simple $A$-module. Then either $Mx = 0$ (if $M$ is $x$-torsion) or the map 
 \[x_{M}:M\rightarrow M\  \text{given by}\  m\mapsto mx\] is invertible (if $M$ is $x$-torsion-free).
\end{lemm}

This lemma implies that the action of a normal element on a simple module is either zero or invertible. Note that $e_3$ is normal in $\mathbf{B}$. In this section, we focus on the construction of $e_3$-torsionfree simple modules over the algebra $\textbf{B}$. By examining the actions of certain central elements of $\textbf{B}$ on these modules, we identify the conditions under which irreducibility is achieved, ultimately leading to the construction of simple modules. The section provides a detailed framework for understanding how these modules are built and their role within the broader algebraic structure.
\subsection{Simple modules of type $\mathcal{V}_1(\alpha,\beta,\gamma,\delta)$}
Given $(\alpha,\beta,\gamma,\delta) \in (\mathbb{K}^*)^3 \times \mathbb{K}$, let $\mathcal{V}_1(\alpha,\beta,\gamma,\delta)$ denote the $\mathbb{K}$ vector space with basis $\{v_k~|~ 0 \leq k \leq l-1\}$. We define the $\textbf{B}$-module structure on $\mathcal{V}_1(\alpha,\beta,\gamma,\delta)$ as follows :
\[v_ke_1=\alpha v_{k\oplus(+1)},~v_ke_3=q^{-2k}\beta v_k,~v_kz=\gamma v_k,~
    v_k\tilde{z}=\alpha^{-1}\left(\delta+\displaystyle\frac{1-q^{-4k}}{1-q^{-4}}\beta^2\right)v_{k\ominus(-1)}\] where $\oplus$ is the sum in $\mathbb{Z}/l\mathbb{Z}$.
We shall see that the above relations indeed define a $\textbf{B}$-module. For $0 \leq k \leq l-1$, we have 
\[v_k\left(e_1e_3-q^{-2}e_3e_1\right)=\alpha \left(v_{k+1}e_3\right)-q^{-2}q^{-2k}\beta \left(v_{k}e_1\right)=q^{-2(k+1)}\alpha\beta v_{k+1}-q^{-2(k+1)}\alpha\beta v_{k+1}=0.\] Similarly we can show 
\[v_k(e_1z-ze_1)=v_k(e_3z-ze_3)=v_k(\tilde{z}z-z\tilde{z})=0.\]
Finally,
\begin{align*}
  v_k\left(\tilde{z}e_1-e_1\tilde{z}+e_3^2\right)&=  \alpha^{-1}\left(\delta+\displaystyle\frac{1-q^{-4k}}{1-q^{-4}}\beta^2\right)\left(v_{k-1}e_1\right)-\alpha \left(v_{k+1}\tilde{z}\right)+q^{-4k}\beta^2v_k\\
  &=\left(\delta+\displaystyle\frac{1-q^{-4k}}{1-q^{-4}}\beta^2\right)v_k-\left(\delta+\displaystyle\frac{1-q^{-4(k+1)}}{1-q^{-4}}\beta^2\right)v_k+q^{-4k}\beta^2v_k\\
  &=0.
\end{align*}
Thus $\mathcal{V}_1(\alpha,\beta,\gamma,\delta)$ is a module over $\textbf{B}$.
\begin{theo}\label{dim}
    The $\textbf{B}$-module $\mathcal{V}_1(\alpha,\beta,\gamma,\delta)$ is a simple module of dimension $l$.
\end{theo}
\begin{proof}
    Let $U$ be any non-zero submodule of $\mathcal{V}_1(\alpha,\beta,\gamma,\delta)$. It is enough to show that $U$ contains a basis vector $v_k$. If so, then as $e_1$ permutes all the basis vector we can conclude $U=\mathcal{V}_1(\alpha,\beta,\gamma,\delta)$.
    \par We use induction to show that $U$ contains at least one basis element $v_k$ for $0 \leq k \leq l-1$. Let $u \in U$ can be written as 
    \[u=\sum_{i\in \mathcal{I}}\xi_iv_i\] where $\mathcal{I}\subseteq \{0,1,\cdots,l-1\}$. If $\mathcal{I}$ is a singleton then we are done. If not, for two distinct $i_1,i_2\in \mathcal{I}$ we have $v_{i_1}$ and $v_{i_2}$ are eigenvectors of the operator $e_3$ corresponding to two distinct eigenvalues $q^{-2(i_1+1)}\beta$ and $q^{-2(i_1+1)}\beta$ respectively. So $ue_3-q^{-2(i_1+1)}\beta u$ is a non-zero element in $U$ with length smaller than $u$. Hence by induction the result follows.
    \end{proof}
    \textbf{Simple modules of type $\mathcal{V}_2(\alpha,\beta,\gamma)$: } Given $(\alpha,\beta,\gamma) \in (\mathbb{K}^*)^3$, let $\mathcal{V}_2(\alpha,\beta,\gamma)$ denotes the $\mathbb{K}$ vector space with basis $\{v_k~|~ 0 \leq k \leq l-1\}$. We define the $\textbf{B}$-module structure on $\mathcal{V}_2(\alpha,\beta,\gamma)$ as follows :
\[v_ke_1=\begin{cases}
    \alpha^{-1}\beta^2\displaystyle\frac{q^{4k}-1}{1-q^{4}} v_{k\oplus(-1)},& k\neq0\\
    0,&k=0
\end{cases},~v_ke_3=q^{2k}\beta v_k,~v_kz=\gamma v_k,~
    v_k\tilde{z}=\alpha v_{k\oplus(+1)}\] where $\oplus$ is the sum in $\mathbb{Z}/l\mathbb{Z}$. In order to establish the well-definedness we need to check that the $\mathbb{K}$-endomorphisms of $\mathcal{V}_2(\alpha,\beta,\gamma)$ defined by the above rules satisfy the relations. Indeed with the above actions we have the following computation:
    \begin{align*}
 \text{For $k\neq 0$,}~~ v_k\left(\tilde{z}e_1-e_1\tilde{z}+e_3^2\right)&=  \alpha \left(v_{k+1}e_1\right)-\alpha^{-1}\displaystyle\frac{q^{4k}-1}{1-q^{4}}\beta^2\left(v_{k-1}\tilde{z}\right)+q^{4k}\beta^2v_k\\
  &=\displaystyle\frac{q^{4(k+1)}-1}{1-q^{4}}\beta^2v_k-\displaystyle\frac{q^{4k}-1}{1-q^{4}}\beta^2v_k+q^{4k}\beta^2v_k\\
  &=0.
\end{align*}
For $k=0$ we have:
\begin{align*}
v_0\left(\tilde{z}e_1-e_1\tilde{z}+e_3^2\right)&=\alpha(v_1e_1)-0+\beta^2v_0\\
    &=-\beta^2v_0+\beta^2v_0=0.
\end{align*}
The other relations are easy to compute. Thus $\mathcal{V}_2(\alpha,\beta,\gamma)$ is a $\textbf{B}$-module.
\begin{theo}\label{dim2}
    The $\textbf{B}$-module $\mathcal{V}_2(\alpha,\beta,\gamma)$ is a simple module of dimension $l$.
\end{theo}
\begin{proof}
    Each of the vector $v_k \in \mathcal{V}_2(\alpha,\beta,\gamma) $ is eigenvector to the operator $e_3$ associated with the eigenvalue $ q^{2k}\beta$. With this fact, the proof is parallel to the proof of Theorem \ref{dim}.
\end{proof}
\textbf{Simple modules type $\mathcal{V}_3(\alpha,\beta)$ :} Given $(\alpha,\beta) \in (\mathbb{K}^*)^2$, let $\mathcal{V}_3(\alpha,\beta)$ denotes the $\mathbb{K}$ vector space with basis $\{v_k~|~ 0 \leq k \leq \ord(q^4)-1\}$. We define the $\textbf{B}$-module structure on $\mathcal{V}_3(\alpha,\beta)$ as follows :
\[v_ke_1=\begin{cases}
    \alpha^2\displaystyle\frac{q^{4k}-1}{1-q^{4}} v_{k-1},& k\neq0\\
    0,&k=0
\end{cases},~v_ke_3=q^{2k}\alpha v_k,~v_kz=\beta v_k,~
    v_k\tilde{z}=\begin{cases}
       v_{k+1},& k\neq \ord(q^4)-1\\
       0,&k=\ord(q^4)-1
    \end{cases}\]
    We can easily show that as in $\mathcal{V}_2(\alpha,\beta,\gamma)$ the actions on $\mathcal{V}_3(\alpha,\beta)$ is indeed a $\textbf{B}$-module.
    \begin{theo}
        The $\textbf{B}$ module $\mathcal{V}_3(\alpha,\beta)$ is a simple module of dimension $\ord(q^4)$.
    \end{theo}
    \begin{proof}
        Similar to Theorem \ref{dim2}
    \end{proof}
    \section{Classification of $e_3$-torsionfree simple $\textbf{B}$-module}
    In this section, we completely classify all $e_3$-torsionfree simple $\textbf{B}$-modules. Let $N$ be a simple module over $\textbf{B}$. Then by Proposition 2.3, the
$\mathbb{K}$-dimension of $N $ is finite and bounded above by PI-deg($\textbf{B})$. This classification
of simple modules is based on the action of appropriate central or normal elements
within the algebra $\textbf{B}$. It is important to note that the element $e_3$ is a normal
element. In view of Lemma 2.1, the action of each normal element 
on a simple module $N$ is either trivial or invertible. Based on this fact, we will now
consider the $e_3$-torsionfree simple modules. \\
Observe that each of the element
\begin{equation}\label{q}
    e_1^{l},\tilde{z}^l,e_3, z , \tilde{z}e_1
\end{equation} commute in $\textbf{B}$. Since $N$ is finite-dimensional, there is a common eigenvector $v$ in $N$ of the operators (\ref{q}) such that 
\begin{equation}\label{q2}
ve_1^{l}=\eta_1v,\ v\tilde{z}^l=\eta_2v,\ ve_3=\eta_3v, \ vz=\eta_4v,\ v\tilde{z}e_1=\eta_5v,\     
\end{equation}
for some $\eta_1,\eta_2,\eta_5 \in \mathbb{K}$ and $\eta_3, \eta_4\in\mathbb{K}^*$. Then the following cases arise depending on the scalars $\eta_1$ and $\eta_2$.\\
\textbf{Case I:} Suppose that $\eta_1 \neq 0$. Then the set $\{ve_{1}^k:0 \leq k \leq l-1\}$ consists of nonzero vectors of $N$. Let us choose
\[\alpha:=\eta_1^{\frac{1}{l}},\ \beta:=\eta_3,\ \gamma:=\eta_4, \ \delta=\eta_5 \] so that 
$(\alpha,\beta,\gamma, \delta)\in {(\mathbb{K}^*)}^3  \times \mathbb{K}$. Note the $\mathbb{K}$-linear map \[\Psi_1:\mathcal{V}_1(\alpha,\beta,\gamma,\delta)\rightarrow N\] defined by
\[\Psi_1(v_k):=\alpha^{-k}ve_1^k\] is a non zero $\textbf{B}$-module homomorphism. The following will be useful to verify 
\begin{align*}
    \left(ve_1^k\right)\tilde{z}=\begin{cases}
        \left(\eta_5+\displaystyle\frac{1-q^{-4k}}{1-q^{-4}}\eta_3^2\right)(ve_1^{k-1}), & k \neq 0\\
        \eta_1^{-1}\eta_5\left(ve_1^{l-1}\right), &k=0
    \end{cases}
\end{align*}
Since $\mathcal{V}_1(\alpha,\beta,\gamma,\delta)$ and $N$ are both simple $\textbf{B}$-modules, by Schur's lemma $\Psi$ is an $\textbf{B}$-module isomorphism.\\
\textbf{Case II:} $\eta_1=0$ and $\eta_2 \neq 0$. Since $\eta_1 =0$ there exists $0\leq p\leq l-1$ such that $ve_1^{p}\neq 0$ and $ve_1^{p+1}=0$. Define $u:=ve_1^{p}\neq 0$. Then from (\ref{q2}) we obtain \[u\tilde{z}^l=\eta_2u, \ ue_3=q^{-2p}\eta_3u, \ uz=\eta_4u.\]
Then the set $\{v\tilde{z}^k:0 \leq k \leq l-1\}$ consists of nonzero vectors of $N$. Let us choose
\[\alpha:=\eta_2^{\frac{1}{l}},\ \beta:=q^{-2p}\eta_3,\ \gamma:=\eta_4\] so that 
$(\alpha,\beta,\gamma)\in {(\mathbb{K}^*)}^3$. Note the $\mathbb{K}$-linear map \[\Psi_2:\mathcal{V}_2(\alpha,\beta,\gamma)\rightarrow N\] defined by
\[\Psi_2(v_k):=\alpha^{-k}u\tilde{z}^k\] is a non zero $\textbf{B}$-module homomorphism. Again the following will be helpful to verify the homomorphism
\begin{align*}
    \left(u\tilde{z}^k\right)e_1=\begin{cases}
        -\displaystyle\frac{1-q^{4k}}{1-q^{4}}q^{-4p}\eta_3^2\left(u\tilde{z}^{k-1}\right), & k \neq 0\\
        0, & k=0.
    \end{cases}
\end{align*}
By Schur's lemma we conclude that $\Psi_2$ is a $\textbf{B}$-module isomorphism.\\
\textbf{Case III:} $\eta_1=0$ and $\eta_2=0$. Since $\eta_1 =0$ there exists $0\leq p\leq l-1$ such that $ve_1^{p}\neq 0$ and $ve_1^{p+1}=0$. Define $u:=ve_1^{p}\neq 0$. Then from (\ref{q2}) we obtain \[u\tilde{z}^l=\eta_2u, \ ue_3=q^{-2p}\eta_3u, \ uz=\eta_4u.\] Since $\eta_2=0$, let $r$ be the smallest integer with $1 \leq r \leq l$ such that $u\tilde{z}^{r-1} \neq 0$ and $u\tilde{z}^r=0$. Then after simplifying the equality $(u\tilde{z}^r)e_1=0$ we have 
\begin{align*}
    0=u\left(e_1\tilde{z}^r-\displaystyle\frac{1-q^{4r}}{1-q^4}e_3^2\tilde{z}^{r-1}\right)=-\displaystyle\frac{1-q^{4r}}{1-q^4}q^{-4p}\eta_3^2\left(u\tilde{z}^{r-1}\right)
\end{align*}
This implies $\ord(q^4)\divides r$. We claim that $r=\ord(q^4)$. 
Note if $\ord(q)$ is odd or $\ord(q) \in 2\mathbb{Z}$ but $\ord(q) \notin 4\mathbb{Z}$, it is obvious.\\
Let $\ord(q) \in 4\mathbb{Z}$. Then $l=2\ord(q^4)$.
If $r \neq \ord(q^4)$ then $r=2\ord(q^4)=l$. Then the subspace of $N$ generated by the set $S:=\{u\tilde{z}^k: \ord(q^4) \leq k \leq 2\ord(q^4)-1\}$ consists of nonzero vectors of $N$. Note 
\begin{align*}
    \left(u\tilde{z}^k\right)e_1&=\begin{cases}-\displaystyle\frac{1-q^{4k}}{1-q^4}q^{-4p}\eta_3^2\left(u\tilde{z}^{k-1}\right), &k\neq \ord(q^4)\\
0, &k=\ord(q^4).
    \end{cases}\\
     \left(u\tilde{z}^k\right)\tilde{z}&=\begin{cases}
         \left(u\tilde{z}^{k+1}\right), & k \neq 2\ord(q^4)-1\\
         0,& k=2\ord(q^4)-1
     \end{cases}\\
\left(u\tilde{z}^k\right)e_3&=q^{2k}q^{-2p}\eta_3\left(u\tilde{z}^k\right)\\
\left(u\tilde{z}^k\right)z&=\eta_4\left(u\tilde{z}^k\right). 
\end{align*}
This is also a proper submodule of $N$. First we claim that the vectors in $S$ are linearly independent. Note for $k \neq k'$, the vectors $u\tilde{z}^k$ and $u\tilde{z}^{k'}$ are eigenvectors of $e_3$ corresponding to distinct eigenvalues. Then we can use induction to show that $S$ is linearly independent. Now we note that $S \cup \{u\}$ is also linearly independent. Indeed let
\[k:=c_1u+c_2\left(u\tilde{z}^{\ord{q^4}}\right)+c_3\left(u\tilde{z}^{\ord{(q^4)}+1}\right)+\cdots+c_l\left(u\tilde{z}^{l-1}\right)=0.\] Note that 
\[0=ke_3-\eta_3k=\sum_{p=\ord(q^4)}^{l-1}c_p\eta_3(q^{2k}-1)\left(u\tilde{z}^k\right)\] implies $c_p=0$ for all $p$ with $\ord(q^4)\leq p \leq l-1$. Hence the proof.\\
Hence the set $\{u\tilde{z}^k:0 \leq k \leq \ord(q^4)-1\}$ consists of nonzero vectors of $N$. Let us choose
\[\alpha:=q^{-2p}\eta_3,\ \beta:=\eta_4\] so that 
$(\alpha,\beta)\in {(\mathbb{K}^*)}^2$. Note the $\mathbb{K}$-linear map \[\Psi_3:\mathcal{V}_3(\alpha,\beta)\rightarrow N\] defined by
\[\Psi_3(v_k):=u\tilde{z}^k\] is a non zero $\textbf{B}$-module homomorphism.
By Schur's lemma we conclude that $\Psi_3$ is a $\textbf{B}$-module isomorphism.
\par Finally the above discussions lead us to one of the main results of this section which provides an opportunity for the classification of simple $U_q^+(B_2)$-modules in terms of scalar parameters.
\begin{theo}\label{main1}
Let $q$ be a primitive $m$-th root of unity. Then each simple $e_3$-torsionfree $\textbf{B}$-module  is isomorphic to one of the following simple $\textbf{B}$-modules:
\begin{itemize}
    \item $\mathcal{V}_1(\alpha,\beta,\gamma,\delta)$ for some $(\alpha,\beta,\gamma,\delta) \in (\mathbb{K}^*)^3 \times \mathbb{K}$,
    \item $\mathcal{V}_2(\alpha,\beta,\gamma)$ for some $(\alpha,\beta,\gamma) \in (\mathbb{K}^*)^3$,
    \item $\mathcal{V}_3(\alpha,\beta)$ for some $(\alpha,\beta)\in (\mathbb{{K}^*})^2$.
\end{itemize}
\end{theo}
\section{Simple $U_q^{+}(B_2)$-modules on which $e_3$ acts invertibly}
We begin this section by establishing the relationship between simple $U_q^{+}(B_2)$-modules on which $e_3$ acts invertible and $e_3$-torsionfree simple $\mathbf{B}$-modules.
\begin{theo}\label{itd}
There is a one to one correspondence between the simple $e_3$-torsionfree $\mathbf{B}$-modules and the simple $U_q^{+}(B_2)$-modules on which $e_3$ acts invertibly.
\end{theo}
\begin{proof}
   From Proposition $\ref{PI}$ along with Theorem $\ref{brg}$, it is quite clear that each simple $U_q^{+}(B_2)$-module is finite dimensional and can have dimension at most $\pideg(U_q^{+}(B_2))$. Same is true for algebra $\mathbf{B}$ also.
\par Let $N$ be a simple $e_3$-torsionfree $\mathbf{B}$-module. Then the operator $e_3$ on $N$ is invertible (as $e_3$ is normal in $\mathbf{B}$). With this fact and the expression of $\tilde{z}$  in (\ref{cen1}), one can define the action of $e_2$ on $N$ explicitly. Note for each $n\in N$ we have 
\[ne_2=n\left(\tilde{z}-\displaystyle\frac{z}{q^2-1}\right)e_3^{-1}\]Thus $N$ becomes an $U_q^{+}(B_2)$-module which is simple as well. Hence any simple $e_3$-torsionfree $\mathbf{B}$-module is a simple $U_q^{+}(B_2)$-module on which $e_3$ acts invertibly.
\par On the other hand suppose $M$ be a simple $U_q^{+}(B_2)$-module on which $e_3$ acts invertibly. Now we claim that $M$ is a simple $e_3$-torsionfree $\mathbf{B}$-module. Clearly $M$ is a finite dimensional $e_3$-torsionfree $\mathbf{B}$-module. Suppose $M'$ is a nonzero simple $\mathbf{B}$-submodule of $M$. As $e_3$ is an invertible operator on $M'$, therefore one can define the action of $e_2$ on $M'$. Thus $M'$ becomes an $U_q^{+}(B_2)$-module with $M'\subseteq M$. Therefore $M=M'$ and hence $M$ is a simple $e_3$-torsionfree $\mathbf{B}$-module.  
\end{proof}
 Now we are ready to classify all possible simple $U_q^{+}(B_2)$-modules on which $e_3$ acts invertibly.
\textbf{Simple modules of type ${\mathcal{V}_1}'(\alpha,\beta,\gamma,\delta)$}
For $(\alpha,\beta,\gamma,\delta) \in (\mathbb{K}^*)^3 \times \mathbb{K}$, ${\mathcal{V}_1}'(\alpha,\beta,\gamma,\delta)$ is a simple $U_q^+(B_2)$-module with basis $\{v_k~|~ 0 \leq k \leq l-1\}$ and actions as follows :
\[v_ke_1=\alpha v_{k+1},~v_ke_2=\alpha^{-1}\left(\delta+\displaystyle\frac{1-q^{-4k}}{1-q^{-4}}\beta^2\right)q^{2(k-1)}\beta^{-1}v_{k-1}-\displaystyle\frac{\gamma}{q^2-1}\beta^{-1}v_k\]
\[v_ke_3=q^{-2k}\beta v_k,~v_kz=\gamma v_k.
    \]
    \textbf{Simple modules of type ${\mathcal{V}_2}'(\alpha,\beta,\gamma)$: } For $(\alpha,\beta,\gamma) \in (\mathbb{K}^*)^3$, ${\mathcal{V}_2}'(\alpha,\beta,\gamma)$ is a simple $U_q^+(B_2)$-module with basis $\{v_k~|~ 0 \leq k \leq l-1\}$ and actions as follows :
\[v_ke_1=\begin{cases}
    \alpha^{-1}\beta^2\displaystyle\frac{q^{4k}-1}{1-q^{4}} v_{k-1},& k\neq0\\
    0,&k=0
\end{cases},v_ke_2=\alpha\beta^{-1}q^{-2(k+1)}v_{k+1}-\displaystyle\frac{\gamma\beta^{-1}}{q^2-1}q^{-2k}v_k,\]\[v_ke_3=q^{2k}\beta v_k,~v_kz=\gamma v_k.
    \] 
    \textbf{Simple modules type ${\mathcal{V}_3}'(\alpha,\beta)$ :} For $(\alpha,\beta) \in (\mathbb{K}^*)^2$, ${\mathcal{V}_3}'(\alpha,\beta)$ is a simple $U_q^+(B_2)$-module with basis $\{v_k~|~ 0 \leq k \leq \ord(q^4)-1\}$ and action as follows :
\[v_ke_1=\begin{cases}
    \alpha^2\displaystyle\frac{q^{4k}-1}{1-q^{4}} v_{k-1},& k\neq0\\
    0,&k=0
\end{cases},v_ke_2=\begin{cases}
    \alpha^{-1}q^{-2(k+1)}v_{k+1}-\displaystyle\frac{\beta\alpha^{-1}}{q^2-1}q^{-2k}v_k, &k\neq \ord(q^4)-1\\
    -\displaystyle\frac{\beta\alpha^{-1}}{q^2-1}q^{-2k}v_k, &k= \ord(q^4)-1
\end{cases}\]\[v_ke_3=q^{2k}\alpha v_k,~v_kz=\beta v_k.\]
\section{Simple $U_q^{+}(B_2)$-modules on which $e_3$ acts nilpotently}
We have already obtained all possible simple $U_q^+(B_2)$-modules on which $e_3$ acts invertibly. In this section our main focus will be on constructing and classifying simple $U_q^+(B_2)$-modules on which $e_3$ acts nilpotently.
\subsection{Construction:} Given $(\alpha,\beta,\gamma)\in \mathbb{K^*}\times {\mathbb{K}}^2$, let ${\mathcal{V}_4}'(\alpha,\beta,\gamma)$ denote the $\mathbb{K}$ vector space with basis $\{v_k~|~ 0 \leq k \leq l-1\}$. We define the $U_q^+(B_2)$-module structure on ${\mathcal{V}_4}'(\alpha,\beta,\gamma)$ as follows :
\[v_ke_1=\begin{cases}
    \beta q^{-2k}v_k-q^{-2(k-1)}\displaystyle\frac{(q^{2k}-1)(q^{2(k-1)}-1)}{(q^4-1)(q^{2}-1)}\alpha v_{k-2},&k \geq 2\\
    \beta q^{-2k}v_k, &k=0,1
\end{cases}, ~v_ke_2=\begin{cases}
    v_{k+1}, &k\neq l-1\\
    \gamma v_0, & k=l-1
\end{cases}\]
\[v_ke_3=\begin{cases}
    \alpha \displaystyle\frac{q^{2k}-1}{q^2-1}v_{k-1},& k \neq 0\\
    0, &k=0
\end{cases},~v_kz=\alpha v_k.\]
These actions indeed define a $U_q^+(B_2)$-module. See Appendix for further calculations.
\begin{theo}
        The $U_q^+(B_2)$ module ${\mathcal{V}}'_4(\alpha,\beta,\gamma)$ is a simple module of dimension $l$.
    \end{theo}
    \begin{proof}
       Notice that each of the vector $v_k \in {\mathcal{V}}'_4(\alpha,\beta,\gamma) $ is eigenvector to the operator $e_3e_2$ associated with the eigenvalue $ \alpha\displaystyle\frac{q^{2k}-1}{q^2-1}$. With this fact, the proof is parallel to the proof of Theorem \ref{dim}.
    \end{proof}
\subsection{Classification:} Let $N'$ be a simple $U_q^+(B_2)$-module on which $e_3$ acts nilpotently. Observe that each of the element
\begin{equation}\label{eq1}
    e_2^{l},e_3^l,e_1, z 
\end{equation} commute in $U_q^+(B_2)$. Since $N'$ is finite-dimensional, there is a common eigenvector $v$ in $N'$ of the operators (\ref{eq1}) such that 
\begin{equation}\label{coev}
ve_2^{l}=\eta_1v,\ ve_3^l=\eta_2v,\ ve_1=\eta_3v, \ vz=\eta_4v     
\end{equation}
for some $\eta_1,\eta_2,\eta_3 \in \mathbb{K}$ and $ \eta_4\in\mathbb{K}^*$. As $e_3$ acts nilpotently $\eta_2=0$. Since $\eta_2 =0$ there exists $0\leq p\leq l-1$ such that $ve_3^{p}\neq 0$ and $ve_3^{p+1}=0$. Define $u:=ve_3^{p}\neq 0$. Then from (\ref{coev}) we obtain \[ue_2^l=\eta_1u, \ ue_1=q^{2p}\eta_3u, \ uz=\eta_4u.\] If $\eta_1 \neq 0$ then the set $\{ue_2^k:0 \leq k \leq l-1\}$ consists of nonzero vectors of $N'$. If $\eta_2=0$, then there exists $0\leq r\leq l-1$ such that $ue_2^{r-1}\neq 0$ and $ue_2^{r}=0$. Then 
\[0=\left(ue_2^r\right)e_3=u\left(q^{2r}e_3e_2^r+\displaystyle\frac{q^{2r}-1}{q^2-1}ze_2^{r-1}\right)=\displaystyle\frac{q^{2r}-1}{q^2-1}\eta_4\left(ue_2^{r-1}\right)\] Hence $r=l$.
So in any case the set $\{ue_2^k:0 \leq k \leq l-1\}$ consists of nonzero vectors of $N'$.  Let us choose
\[\alpha:=\eta_4,\ \beta:=q^{2p}\eta_3,\ \gamma:=\eta_1\] so that 
$(\alpha,\beta,\gamma)\in {(\mathbb{K}^*)} \times {\mathbb{K}}^2$. Note the $\mathbb{K}$-linear map \[\Psi_4:{\mathcal{V}}'_4(\alpha,\beta,\gamma)\rightarrow N'\] defined by
\[\Psi_4(v_k):=ue_2^k\] is a non zero $U_q^+(B_2)$-module homomorphism. By Schur's lemma we conclude that $\Psi_4$ is a $U_q^+(B_2)$-module isomorphism. Hence, we conclude this section with the following theorem:
\begin{theo}
    Let $N'$ be any simple $U_q^+(B_2)$-module on which $e_3$ acts nilpotently. Then $N'$ is isomorphic to ${\mathcal{V}}'_4(\alpha,\beta,\gamma)$ for some $(\alpha,\beta,\gamma)\in {(\mathbb{K}^*)} \times {\mathbb{K}}^2$.
\end{theo}
So we can summarize all the results in the following theorem.
\begin{theo}\label{mainresult}
Any simple $U_q^+(B_2)$-module is isomorphic to one of the following simple $U_q^+(B_2)$-modules
\begin{itemize}
\item[1.] ${\mathcal{V}}'_1(\alpha,\beta,\gamma,\delta)$ for some $(\alpha,\beta,\gamma,\delta) \in (\mathbb{K}^*)^3 \times \mathbb{K}$,
    \item[2.] ${\mathcal{V}}'_2(\alpha,\beta,\gamma)$ for some $(\alpha,\beta,\gamma) \in (\mathbb{K}^*)^3$,
    \item[3.] ${\mathcal{V}}'_3(\alpha,\beta)$ for some $(\alpha,\beta)\in (\mathbb{{K}^*})^2$.
    \item[4.] ${\mathcal{V}}'_4(\alpha,\beta,\gamma)$ for some $(\alpha,\beta,\gamma)\in {(\mathbb{K}^*)} \times {\mathbb{K}}^2$.
\end{itemize}
\end{theo}
\begin{rema}\label{mainremark}
    We have the following observations.
\begin{enumerate}
\item $e_1^l, e_3^l$ does not annihilate the simple $U_q^+(B_2)$-module ${\mathcal{V}}'_1(\alpha,\beta,\gamma,\delta)$.
\item $e_1^l$ annihilates the simple $U_q^+(B_2)$-module ${\mathcal{V}}'_2(\alpha,\beta,\gamma)$, but $\tilde{z}^l,e_3^l$ does not.
\item Both $e_1^l$ and $\tilde{z}^l$ annihilate the simple $U_q^+(B_2)$-module${\mathcal{V}}'_3(\alpha,\beta)$ but $e_3^l$ does not.
\item $e_3^l$ annihilates $U_q^+(B_2)$-module ${\mathcal{V}}'_4(\alpha,\beta,\gamma)$.
\end{enumerate}
\end{rema}
\section{Isomorphism classes of simple $U_q^+(B_2)$-module}
In this section, we investigate the conditions under which two modules, as classified in Theorem \ref{mainresult}, are isomorphic. Remark \ref{mainremark} implies that modules from different types, as described in the theorem, cannot be isomorphic to one another. However, it remains possible for two distinct modules of the same type to be isomorphic.
\begin{theo}\label{iso3}
Let $(\alpha_1,\beta_1,\gamma_1,\delta_1)$ and $(\alpha_2,\beta_2,\gamma_2,\delta_2)$ be in $(\mathbb{K}^*)^3 \times \mathbb{K}$.
Then ${\mathcal{V}}'_1(\alpha_1,\beta_1,\gamma_1,\delta_1)$ is isomorphic to ${\mathcal{V}}'_1(\alpha_2,\beta_2,\gamma_2,\delta_2)$ as $U_q^+(B_2)$-modules if and only if
\[\alpha_1^l=\alpha_2^l, \ \beta_1=q^{-2p}\beta_2, \ \gamma_1=\gamma_2, \ \delta_1=\delta_2+\displaystyle\frac{1-q^{-4p}}{1-q^{-4}}\beta_2^2\]
holds for some $p$ such that $0 \leq p \leq l-1$. 
\end{theo}
\begin{proof}
    Let $\psi:{\mathcal{V}}'_1(\alpha_1,\beta_1,\gamma_1,\delta_1) \rightarrow {\mathcal{V}}'_1(\alpha_2,\beta_2,\gamma_2,\delta_2)$ be a $U_q^+(B_2)$-module isomorphism. As \[v_k=\alpha_1^{-a}v_0e_1^k\] holds in ${\mathcal{V}}'_1(\alpha_1,\beta_1,\gamma_1,\delta_1)$, therefore $\psi$ is uniquely determined by $\psi(v_0)$. Suppose 
\begin{equation}\label{isoeqn}
\psi\left(v_0\right)=\sum_{p}\lambda_{p}v_p
\end{equation}
where $0 \leq p\leq l-1$ and at least one $\lambda_{p}\in \mathbb{K}^*$. If possible let there are two non-zero scalars $\lambda_{u}$ and $\lambda_{v}$ in (\ref{isoeqn}). Now equating the coefficients of basis vectors on both sides of the equalities 
\[\psi(v_0e_3)=\psi(v_0)e_3\] we have $$\sum_{p}\lambda_p\left(\beta_1-q^{-2p}\beta_2\right)v_p=0$$. This implies $\beta_1=q^{-2u}\beta_2=q^{-2v}\beta_2$, which is a contradiction. Thus $\psi\left(v_0\right)=\lambda_{p}v_p$ for some $\lambda_{p}\in \mathbb{K}^*$. This will help us to determine the relationship between the scalars. The actions of $e_1^{l},\ e_3$, $z$ and $\tilde{z}e_1$ under the isomorphism $\psi$ give us 
    \[\alpha_1^l=\alpha_2^l, \ \beta_1=q^{-2p}\beta_2, \ \gamma_1=\gamma_2, \ \delta_1=\delta_2+\displaystyle\frac{1-q^{-4p}}{1-q^{-4}}\beta_2^2\]
    \par Conversely, assume that the relations between $(\alpha_1,\beta_1,\gamma_1,\delta_1)$ and $(\alpha_2,\beta_2,\gamma_2,\delta_2)$ hold. Then define a $\mathbb{K}$-linear map $\phi:{\mathcal{V}}'_1(\alpha_1,\beta_1,\gamma_1,\delta_1) \rightarrow {\mathcal{V}}'_1(\alpha_2,\beta_2,\gamma_2,\delta_2)$ by 
\[\phi(v_k)=(\alpha_1^{-1}\alpha_2)^{k}v_{k+p}\] It can be easily checked that $\phi$ is a $U_q^+(B_2)$-isomorphism.
\end{proof}
Similarly we can have the following theorems.
\begin{theo}
    Let $(\alpha_1,\beta_1,\gamma_1)$ and $(\alpha_2,\beta_2,\gamma_2)$ be in $(\mathbb{K}^*)^3$.
Then ${\mathcal{V}}'_2(\alpha_1,\beta_1,\gamma_1)$ is isomorphic to ${\mathcal{V}}'_2(\alpha_2,\beta_2,\gamma_2)$ as $U_q^+(B_2)$-modules if and only if
\[\alpha_1^l=\alpha_2^l, \ \beta_1=q^{2p}\beta_2, \ \gamma_1=\gamma_2\]
holds for some $p$ such that $0 \leq p \leq l-1$. 

\end{theo}
\begin{theo}
    Let $(\alpha_1,\beta_1)$ and $(\alpha_2,\beta_2)$ be in $(\mathbb{K}^*)^2 $.
Then ${\mathcal{V}}'_3(\alpha_1,\beta_1)$ is isomorphic to ${\mathcal{V}}'_3(\alpha_2,\beta_2)$ as $U_q^+(B_2)$-modules if and only if
\[\alpha_1=q^{2p}\alpha_2, \ \beta_1=\beta_2\]
holds for some $p$ such that $0 \leq p \leq l-1$. 
\end{theo}
\begin{theo}
  Let $(\alpha_1,\beta_1,\gamma_1)$ and $(\alpha_2,\beta_2,\gamma_2)$ be in $(\mathbb{K}^*) \times (\mathbb{K})^2$.
Then ${\mathcal{V}}'_4(\alpha_1,\beta_1,\gamma_1)$ is isomorphic to ${\mathcal{V}}'_4(\alpha_2,\beta_2,\gamma_2)$ as $U_q^+(B_2)$-modules if and only if
\[\alpha_1=\displaystyle\frac{q^{2(p+1)}-1}{q^2-1}\alpha_2, \ \beta_1=q^{-2p}\beta_2, \ \gamma_1=\gamma_2\]
holds for some $p$ such that $0 \leq p \leq l-1$.    
\end{theo}
\section{Center of $U_q^+(B_2)$}

 By Proposition \ref{GWA} we conclude that $\mathbf{B}$ is a GWA and $\mathbf{B}=\mathbb{K}[z,e_3][e_1,\tilde{z};\sigma,\rho=1,b=e_3^2]$ where $\sigma(z)=z$ and $\sigma(e_3)=q^{-2}e_3$. The element $\displaystyle\frac{e_3^2}{1-q^{-4}} \in \mathbb{K}[z,e_3]$ is a solution of equation $\alpha-\sigma(\alpha)=e_3^2$. Hence, by Proposition \ref{GWA}, the element \begin{equation}\label{cen3}
     z_1:=e_1\tilde{z}+\frac{e_3^2}{q^4-1}
 \end{equation} is a central element of $\mathbf{B}$ i.e. $z_1$ commutes with $e_1,e_3$ and $z$. Also observe that $z_1e_2=e_2z_1$. So $z_1$ is also central in $U_q^+(B_2)$.
Let $\mathcal{A}$ denote the quantum affine space generated by $X_1$, $X_2$, $X_3$ and $Z$, subject to the relations:
 \[X_2X_1=q^{-2}X_2X_1, \ X_3X_1=q^{2}X_1X_3, \ X_3X_2=q^{-2}X_2X_3, \ X_iZ=ZX_i \ \forall i =1,2,3.\] We refer to $\mathcal{A}$ as the associated quasipolynomial algebra of $U_q^+(B_2)$. We define an ordering on the generators of $U_q^+(B_2)$ and $\mathcal{A}$ by $e_1<e_2<e_3<z$ and $X_1<X_2<X_3<X_4$, respectively. Let $Z\left(U_q^+(B_2)\right)$ be the center of $U_q^+(B_2)$, and $\mathcal{Z}$ be the center of $\mathcal{A}$. For any $C \in Z\left(U_q^+(B_2)\right)$, the leading term of $C$ must be of the form $ke_1^ae_2^be_3^cz^d$, where $k \in \mathbb{K}^*$ and $a,b,c, d \in \mathbb{Z}_{\geq 0}$. Since $e_1C=Ce_1$ holds, equating the coefficients of leading terms on both sides yields:
\begin{equation}\label{eq}
q^{2(c-b)}=1.
\end{equation}
Similarly, the equalities $e_2C=Ce_2$ and $e_3C=Ce_3$ give:
\begin{equation}\label{eq2}
q^{2(a-c)}=q^{2(a-b)}=1.
\end{equation}
From equations \eqref{eq} and \eqref{eq2}, we conclude that the monomial $kX_1^aX_2^bX_3^cZ^d$ in $\mathcal{A}$ corresponding to the leading term of $C$ becomes a central element in $\mathcal{A}$. Hence we can define a map \[T:Z\left(U_q^+(B_2)\right) \rightarrow \mathcal{Z}\] such that \[T(C)=kX_1^aX_2^bX_3^cZ^d\] where $ke_1^ae_2^be_3^cz^d$ is the leading term of $C\in Z\left(U_q^+(B_2)\right)$. Note that $T$ is not linear and $T(0)=0$.
\par In the following we first compute the center of the quantum affine space $\mathcal{A}$ using the key result \cite[Proposition 7.1(a)]{di}. We then focus on the center of $U_q^+(B_2)$ by applying the fiber of the map $T$.
\begin{prop}\label{qascen}
Suppose that $q^2$ is a primitive $l$-th root of unity. Then the center $\mathcal{Z}$ of $\mathcal{A}$ is generated by $X_1^l,X_2^l,X_3^l$, $Z$ and $X_1X_2X_3$.  
 \end{prop}
 \begin{proof}
First, consider $q^2$ as the defining parameter for $\mathcal{A}$. The skew-symmetric integral matrix associated with $\mathcal{A}$, based on the ordering $X_1<X_2<X_3<Z$ of the generators, is given by \[H:=\begin{pmatrix}
    0&1&-1&0\\
    -1&0&1&0\\
    1&-1&0&0\\
    0&0&0&0\\
\end{pmatrix}.\]
We now consider the matrix $H$ as a matrix of a homomorphism $H:\mathbb{Z}^4\rightarrow (\mathbb{Z}/{l\mathbb{Z}})^4$. Then the kernel of $H$ is
\[\ker(H)=\{(a,b,c,d)\in \mathbb{Z}^4:a\equiv b\equiv c~(\text{mod}~ l)\}.\] It is easy to verify that $\ker(H)\cap (\mathbb{Z}_{\geq 0})^4$ is a subsemigroup of $\ker(H)$ generated by $(l,0,0,0)$, $(0,l,0,0),(0,0,l,0)$ $(0,0,0,1)$ and $(1,1,1,0)$. Finally, the assertion follows from \cite[Proposition 7.1(a)]{di}.
 \end{proof}
The central element $z_1$ as in (\ref{cen3}) can be expressed using the ordering on the generators $e_1<e_2<e_3<z$ as \[z_1:=e_1e_2e_3+\frac{e_1z}{q^2-1}+\frac{e_3^2}{q^4-1}.\] 
  \begin{theo}\label{cen}
      The center of $U_q^+(B_2)$ is generated by $e_1^l,e_2^l,e_3^l, z$ and $z_1$. 
  \end{theo}
  \begin{proof}
     Let $Z_1$ denote the subalgebra of $Z\left(U_q^+(B_2)\right)$ generated by elements $e_1^l$, $e_2^l$, $e_3^l$, $z$ and $z_1$. We use induction on the degree of the leading term to show that any $C \in Z\left(U_q^+(B_2)\right)$ is also in $Z_1$. Since $T(C)$ is a monomial that belongs to the center $\mathcal{Z}$ of $\mathcal{A}$, by Proposition \ref{qascen}, it must be of the form \[T(C)=kX_1^{al}X_2^{bl}X_3^{cl}Z^d(X_1X_2X_3)^e,\ \ \text{for}\ k \in \mathbb{K}^*.\] This implies, based on the definition of the map $T$, that the leading term of $C$ will be $k'e_1^{al+e}e_2^{bl+e}e_3^{cl+e}z^d$ for some $k'\in \mathbb{K}^*$. Then we choose an element $C'$ in $Z_1$ of the form \[C':=(e_1^l)^a(e_2^l)^b(e_3^l)^cz^dz_1^e + \text{lower degree terms generated by $e_1^l,e_2^l,e_3^l,z,z_1$}.\] 
Note the leading term of $C'$ becomes $k''e_1^{al+e}e_2^{bl+e}e_3^{cl+e}z^d$ for some $k''\in \mathbb{K}^*$. Then the degree of the leading term of $k''C-k'C'$ is less than that of $C$. Hence, by the induction hypothesis, we obtain $k''C-k'C' \in Z_1$. Since $C' \in Z_1$ and $k''\in \mathbb{K}^*$, it follows that $C\in Z_1$. This completes the proof.
  \end{proof}

\section*{Acknowledgment}The first author expresses gratitude to the Gandhi Institute of Technology and Management for providing an excellent research environment. The second author thanks the Indian Institute of Technology, Kanpur, for the Institute Postdoctoral Fellowship that supported this research.

\section*{Appendix}
\textbf{Verification of $v_k(e_2e_1)=v_k(q^{-2}e_1e_2-q^{-2}e_3)$}:\\
\underline{For $k \neq 0,1,l-1$:} we have
\[v_k(e_2e_1)=v_{k+1}e_1= \beta q^{-2(k+1)}v_{k+1}-q^{-2k}\displaystyle\frac{(q^{2(k+1)}-1)(q^{2k}-1)}{(q^4-1)(q^{2}-1)}\alpha v_{k-1}\]
Again 
\begin{align*}
   & v_k\left(q^{-2}e_1e_2-q^{-2}e_3\right)\\
   &=q^{-2}\left(\beta q^{-2k}v_k-q^{-2(k-1)}\displaystyle\frac{(q^{2k}-1)(q^{2(k-1)}-1)}{(q^4-1)(q^{2}-1)}\alpha v_{k-2}\right)e_2-q^{-2}\alpha \displaystyle\frac{q^{2k}-1}{q^2-1}v_{k-1}\\
   &=\beta q^{-2(k+1)}v_{k+1}-q^{-2k}\displaystyle\frac{(q^{2k}-1)(q^{2(k-1)}-1)}{(q^4-1)(q^{2}-1)}\alpha v_{k-1}-q^{-2}\alpha \displaystyle\frac{q^{2k}-1}{q^2-1}v_{k-1}\\
  &= \beta q^{-2(k+1)}v_{k+1}-q^{-2k}\displaystyle\frac{(q^{2(k+1)}-1)(q^{2k}-1)}{(q^4-1)(q^{2}-1)}\alpha v_{k-1}
\end{align*}
\underline{For $k=\ord(q^2)-1$:} $v_k(e_2e_1)=\gamma v_0e_1=\gamma\beta v_0$.\\
Similarly,
\begin{align*}
  &v_k\left(q^{-2}e_1e_2-q^{-2}e_3\right)\\
  &=q^{-2}\left(\beta q^{-2k}v_k-q^{-2(k-1)}\displaystyle\frac{(q^{2k}-1)(q^{2(k-1)}-1)}{(q^4-1)(q^{2}-1)}\alpha v_{k-2}\right)e_2-q^{-2}\alpha \displaystyle\frac{q^{2k}-1}{q^2-1}v_{k-1}\\
   &=\beta q^{-2(k+1)}v_{k+1}-q^{-2k}\displaystyle\frac{(q^{2k}-1)(q^{2(k-1)}-1)}{(q^4-1)(q^{2}-1)}\alpha v_{k-1}-q^{-2}\alpha \displaystyle\frac{q^{2k}-1}{q^2-1}v_{k-1}\\
  &= \beta q^{-2(k+1)}v_{k+1}-q^{-2k}\displaystyle\frac{(q^{2(k+1)}-1)(q^{2k}-1)}{(q^4-1)(q^{2}-1)}\alpha v_{k-1}\\
  &=\gamma\beta v_0.
\end{align*}
\underline{For $k=0$:} $v_k(e_2e_1)=v_0(e_2e_1)v_1e_1=\beta q^{-2}v_1$.\\
Also, $v_0\left(q^{-2}e_1e_2-q^{-2}e_3\right)=\beta q^{-2}v_0e_2=\beta q^{-2}v_1$.\\
\underline{For $k=1$:}
\[v_1(e_2e_1)=v_2e_1=\beta q^{-4}v_2-q^{-2}\displaystyle\frac{(q^{4}-1)(q^2-1)}{(q^{4}-1)(q^2-1)}\alpha v_0=\beta q^{-4}v_2-\alpha q^{-2}v_0\]
Now, \[v_1\left(q^{-2}e_1e_2-q^{-2}e_3\right)=q^{-2}\left(\beta q^{-2}v_1e_2\right)-q^{-2}\left(\alpha v_0\right)=\beta q^{-4}v_2-\alpha q^{-2}v_0.\]
\textbf{Verification of $v_k(e_2e_3)=v_k(q^2e_3e_2+z)$}:\\
\underline{$k \neq 0,\ord(q^2)-1$:}\\
\[v_k(e_2e_3)=v_{k+1}e_3=\alpha\displaystyle\frac{q^{2(k+1)}-1}{q^2-1}v_k\]
\[v_k(q^2e_3e_2+z)=q^2\left(\alpha\displaystyle\frac{q^{2k}-1}{q^2-1}v_{k-1}e_2\right)+\alpha v_k=\alpha\displaystyle\frac{q^{2(k+1)}-1}{q^2-1}v_k\]
\underline{$k=0$:}\\
$v_0(e_2e_3)=v_1e_3=\alpha v_0$.\\
$v_0(q^2e_3e_2+z)=\alpha v_0$.\\
\underline{$k=\ord{q^2}-1$:}\\
$v_k(e_2e_3)=\gamma v_0e_3=0$.\\
$v_k(q^2e_3e_2+z)=q^2\left(\alpha\displaystyle\frac{q^{2k}-1}{q^2-1}v_{k-1}e_2\right)+\alpha v_k=\alpha\displaystyle\frac{q^{2(k+1)}-1}{q^2-1}v_k=0$.\\
\textbf{Verification of $v_k(e_1e_3)=v_k(q^{-2}e_3e_2)$:}\\
\underline{$k \neq 0,1,2$:}\\
\begin{align*}
    v_k(e_1e_3)&=\beta q^{-2k}v_k-q^{-2(k-1)}\displaystyle\frac{(q^{2k}-1)(q^{2(k-1)}-1)}{(q^4-1)(q^{2}-1)}\alpha v_{k-2}\\
    &=\beta\alpha q^{-2k}\displaystyle\frac{q^{2k}-1}{q^2-1}v_{k-1}-q^{-2(k-1)}\displaystyle\frac{(q^{2k}-1)(q^{2(k-1)}-1)(q^{2(k-2)}-1)}{(q^4-1)(q^{2}-1)^2}\alpha^2 v_{k-3}.
\end{align*}
\begin{align*}
    v_k(q^{-2}e_3e_1)&=q^{-2}\left(\alpha \displaystyle\frac{q^{2k}-1}{q^2-1}v_{k-1}e_1\right)\\
    &= \alpha q^{-2}\displaystyle\frac{q^{2k}-1}{q^2-1}\left(\beta q^{-2(k-1)}v_{k-1}-q^{-2(k-2)}\displaystyle\frac{(q^{2(k-1)}-1)(q^{2(k-2)}-1)}{(q^4-1)(q^{2}-1)}\alpha v_{k-3}\right)\\
    &=\beta\alpha q^{-2k}\displaystyle\frac{q^{2k}-1}{q^2-1}v_{k-1}-q^{-2(k-1)}\displaystyle\frac{(q^{2k}-1)(q^{2(k-1)}-1)(q^{2(k-2)}-1)}{(q^4-1)(q^{2}-1)^2}\alpha^2 v_{k-3}.
\end{align*}
\underline{$k=0$:}  $v_0(e_1e_3)=\beta v_0e_3=0$ and $v_0(q^{-2}e_3e_1)=0$.\\
\underline{$k=1$:}  $v_1(e_1e_3)=\beta q^{-2}v_1e_3=\alpha\beta q^{-2}v_0$ and $v_1(q^{-2}e_3e_1)=q^{-2}(\alpha v_0e_1)=\alpha\beta q^{-2}v_0$.\\
\underline{$k=2$:}
\[v_2(e_1e_3)=\left(\beta q^{-4}v_2-q^{-2}\displaystyle\frac{(q^4-1)(q^2-1)}{(q^4-1)(q^2-1)}\alpha v_0\right)e_3=\left(\beta q^{-4}v_2-\alpha q^{-2}v_0\right)e_3=\alpha\beta q^{-4}\displaystyle\frac{q^4-1}{q^2-1}v_1.\]
\[v_2(q^{-2}e_3e_1)=q^{-2}\left(\alpha\displaystyle\frac{q^4-1}{q^2-1}v_1\right)e_1=\alpha\beta q^{-4}\displaystyle\frac{q^4-1}{q^2-1}v_1.\]
\textbf{Verification of $v_k(e_iz)=v_k(ze_i)$ for $i=1,2,3$:} These calculations are trivial as the action of $z$ on $v_k$s are just a scalar multiplication.
\bibliographystyle{plain}

\end{document}